\documentclass[12pt]{amsart}
\newtheorem{theorem}{Theorem}[section]
\newtheorem{lemma}[theorem]{Lemma}
\newtheorem{proposition}[theorem]{Proposition}

 \theoremstyle{definition}

\theoremstyle{remark}
\newtheorem{remark}[theorem]{Remark}

\numberwithin{equation}{section}

\begin{document}
\setlength{\baselineskip}{1.2\baselineskip}

\title[mixed complex Hessian equations]
{The Neumann Problem for a class of mixed complex Hessian equations}
\author{Chuanqiang Chen}
\address{School of Mathematics and Statistics, Ningbo University, Ningbo,  315211, P.R. China}
\email{chenchuanqiang@nbu.edu.cn}

\author{Li Chen}
\address{Faculty of Mathematics and Statistics, Hubei Key Laboratory of Applied Mathematics, Hubei University,  Wuhan 430062, P.R. China}
\email{chernli@163.com}

\author{Ni Xiang}
\address{Faculty of Mathematics and Statistics, Hubei Key Laboratory of Applied Mathematics, Hubei University,  Wuhan 430062, P.R. China}
\email{nixiang@hubu.edu.cn}
\thanks{Research of the first author was supported by NSFC No.11771396, and research of the second and the third authors was supported by funds from Hubei Provincial Department of Education Key Projects D20171004, D20181003 and the National Natural Science Foundation of China No.11971157.}

\begin{abstract}
In this paper, we consider the Neumann problem of  a class of mixed complex Hessian equations $\sigma_k(\partial \bar{\partial} u) = \sum\limits _{l=0}^{k-1} \alpha_l(x) \sigma_l (\partial \bar{\partial} u)$ with $\alpha_l$ positive and $2 \leq k \leq n$, and establish the global $C^1$ estimates and reduce the global second derivative estimate to the estimate of double normal second derivatives on the boundary. In particular, we can prove the global $C^2$ estimates and the existence theorems when $k=n$.



\end{abstract}

\maketitle
\bigskip

\section{Introduction}

\medskip
In this paper, we consider the Neumann problem for the following mixed complex Hessian equations
\begin{equation}\label{1.1}
\sigma_k(\partial \bar{\partial} u) = \sum _{l=0}^{k-1} \alpha_l(z) \sigma_l (\partial \bar{\partial} u),  ~~ z \in \Omega \subset \mathbb{C}^n,
\end{equation}
where $k\geq 2$, $\Omega \subset \mathbb{C}^n$ is a bounded domain, $\partial \bar{\partial} u = \{ \frac{\partial^2 u}{\partial z_i \partial \overline{z_j}} \}_{1 \leq i,j \leq n} $ is the complex Hessian matrix of the real valued function $u$, $\alpha_l(z)>0$ in $\overline{\Omega}$ with $l=0, 1, \cdots, k-1$, are given real valued positive functions in $\overline{\Omega}$, and for any $m = 1, \cdots, n$,
\begin{equation*}
\sigma_m(\partial \bar{\partial} u) = \sigma_m(\lambda(\partial \bar{\partial} u)) = \sum _{1 \le i_1 < i_2 <\cdots<i_m\leq n}\lambda_{i_1}\lambda_{i_2}\cdots\lambda_{i_m},
\end{equation*}
with $\lambda(\partial \bar{\partial} u) =(\lambda_1,\cdots,\lambda_n)$ being the eigenvalues of $\partial \bar{\partial} u$. We also set $\sigma_0=1$. Recall that the G{\aa}rding's cone is defined as
\begin{equation*}
\Gamma_k  = \{ \lambda  \in \mathbb{R}^n :\sigma _i (\lambda ) > 0,\forall 1 \le i \le k\}.
\end{equation*}
If $\lambda(\partial \bar{\partial} u) \in \Gamma_k$ for any $z \in \Omega$, we say $u$ is a $k$-admissible function.

The equation \eqref{1.1} is a general class of mixed Hessian equation. Specially, it is complex Monge-Amp\`{e}re equation when $k=n$, $\alpha_0(z) >0$ and $\alpha_1(z) = \cdots = \alpha_{n-1}(z) \equiv 0$, complex $k$-Hessian equation when $\alpha_0(z) >0$ and $\alpha_1(z) = \cdots = \alpha_{k-1}(z) \equiv 0$, and complex Hessian quotient equation when $\alpha_{m}(z) >0$ ($k-1\geq m>0$) and $\alpha_0(z)= \cdots= \alpha_{m-1}(z)= \alpha_{m+1}(z) = \cdots= \alpha_{k-1}(z) \equiv 0$.  This kind of equations is motivated from the study of many important geometric problems. For example, the problem of prescribing convex combination of area measures was proposed in \cite{S}, which leads to mixed Hessian equations of the form
\[
\sigma _k (\nabla ^2 u + uI_n ) + \sum\limits_{i = 0}^{k - 1} {\alpha _i \sigma _i (\nabla ^2 u + uI_n )}  = \phi (x), x \in  \mathbb{S}^n.
\]
The special Lagrangian equation introduced by Harvey-Lawson \cite{HL82} in the study of calibrated geometries is also a mixed type Hessian equation
\[
{\mathop{\rm Im}\nolimits} \det (I_{2n}  + \sqrt { - 1} D^2 u) = \sum\limits_{k = 0}^{[(n - 1)/2]} {( - 1)^k \sigma _{2k + 1} (D^2 u)}  = 0.
\]
Another important example is Fu-Yau equation in \cite{FY2007,FY2008} arising from the study of the Hull-Strominger system in theoretical physics, which is an equation that can be written as the linear combination of the first and the second elementary symmetric functions
\[
\sigma _1 (i\partial \overline \partial  (e^u  + \alpha 'e^{ - u} )) + \alpha '\sigma _2 (i\partial \overline \partial  u) = 0.
\]

For the Dirichlet problem of elliptic equations in $\mathbb{R}^n$, many results are well known. For example, the Dirichlet problem of the Laplace equation was studied in \cite{GT}. Caffarelli-Nirenberg-Spruck \cite{CNS84} and Ivochkina \cite{I87} solved the Dirichlet problem of the Monge-Amp\`{e}re equation. Caffarelli-Nirenberg-Spruck \cite{CNS85} solved the Dirichlet problem of the $k$-Hessian equation. For the general Hessian quotient equation, the Dirichlet problem was solved by Trudinger in \cite{T95}. Also, the Neumann or oblique derivative problem of partial differential equations  has been widely studied. For a priori estimates and the existence theorem of Laplace equation with Neumann boundary condition, we refer to the book \cite{GT}. Also, we can see the recent book written by Lieberman \cite{L13} for the Neumann or oblique derivative problem of linear and quasilinear elliptic equations. In 1986, Lions-Trudinger-Urbas solved the Neumann problem of the Monge-Amp\`{e}re equation in the celebrated paper \cite{LTU86}. For related results on the Neumann or oblique derivative problem for some class of fully nonlinear elliptic equations can be found in Urbas \cite{U95} and \cite{U96}. For the Neumann problem of $k$-Hessian equations,  Trudinger \cite{T87} established the existence theorem when the domain is a ball, and Ma-Qiu \cite{MQ15} and Qiu-Xia \cite{QX16} solved the strictly convex domain case. D.K. Zhang and the first author \cite{CZ16} solved the Neumann problem of general Hessian quotient equations. Jiang and Trudinger \cite{JT15,JT16,JT3}, studied the general oblique boundary problem for augmented Hessian equations with some regular conditions and concavity conditions.

Krylov in \cite{Kr} considered the Dirichlet problem of real case of \eqref{1.1}, and Guan-Zhang in \cite{GZ2019} considered the $(k-1)$-admissible solution without the sign of $\alpha_{k-1}$ and obtained the global $C^2$ estimates. Recently, the classical Neumann problem of real case of \eqref{1.1} is solved by the authors \cite{CCX19}.

The recent development of the Neumann boundary problem for real equation is a motivation for us to study the corresponding complex equation.

For the complex Monge-Amp\`{e}re equation, the Dirichlet problem is solved by Caffarelli-Kohn-Nirenberg-Spruck \cite{CNSK85}, and the Neumann problem is solved by S.Y. Li \cite{L94}. Recently, W. Wei and the first author obtained part results about the Neumann problem of complex Hessian quotient equations in \cite{CW17}.

Naturally, we want to know how about the Neumann problem of the mixed complex Hessian equation \eqref{1.1}. In this paper, we obtain the existence theorem as follows.
\begin{theorem} \label{th1.1}
Suppose that $\Omega \subset \mathbb{C}^n$ is a $C^4$ strictly convex domain, $2 \leq k \leq n$, $\nu$ is the outer unit normal vector of $\partial \Omega$, $\alpha_l(z) \in C^2(\overline{\Omega})$ with $l=0, 1, \cdots, k-1$ are positive real valued functions and $\varphi \in C^3(\partial \Omega)$ is a real valued function. Moreover, if $u \in C^4(\Omega)\cap C^3(\overline \Omega)$ is the $k$-admissible solution of
\begin{align} \label{1.2}
\left\{ \begin{array}{l}
\sigma _k (\partial \bar{\partial} u) = \sum\limits_{l=0}^{k-1} \alpha_l(z) \sigma_l (\partial \bar{\partial} u),  \quad \text{in} \quad \Omega,\\
D_\nu u = - \varepsilon u + \varphi(z),\qquad \text{on} \quad \partial \Omega,
 \end{array} \right.
\end{align}
for small $ \varepsilon >0$. Then we have
\begin{align*}
\sup_{\Omega} |u|  \leq C_0, \quad \sup_{\Omega} |D u|  \leq C_1,
\end{align*}
and
\begin{align*}
\sup_{\Omega} |D^2 u|  \leq C_2 (1+\max_{\partial \Omega} |D_{\nu \nu} u| ),
\end{align*}
where $C_0$ depends on $n$, $k$, $\varepsilon$, $ \text{diam} (\Omega)$, $\max\limits_{\partial \Omega} |\varphi|$ and $\sum\limits_{l=0}^{k-1} \sup\limits_\Omega \alpha_l$; $C_1$ depends on $n$, $k$, $\Omega$, $|\varphi|_{C^2}$, $\inf_\Omega \alpha_l$ and $|\alpha_l|_{C^1}$; and $C_2$ depends on $n$, $k$, $\Omega$, $|\varphi|_{C^3}$, $\inf_\Omega \alpha_l$ and $|\alpha_l|_{C^2}$.
\end{theorem}

\begin{remark} \label{rmk1.2}
In this paper, we always denote
\begin{align*}
&z = (z_1, \cdots, z_n) \in  \overline{\Omega}, \quad z_j = t_j + \sqrt{-1} t_{n+j}, \quad t = (t_1, \cdots, t_n, t_{n+1}, \cdots, t_{2n}); \\
&\partial_j u = \frac{\partial u}{\partial z_j} = u_{z_j}, \quad \partial_{\overline{j}} u = \frac{\partial u}{\partial \overline{z_j}}= u_{\overline{z_j}}, \quad \partial u = (\partial_1 u, \cdots, \partial_n u). \\
& D_k u = \frac{\partial u}{\partial t_k }, \quad Du=(D_1 u, \cdots, D_{2n} u),
\end{align*}
where $\sqrt{-1}$ is the imaginary unit. It is easy to see
\begin{align*}
&\partial_j u  = \frac{1}{2} [D_j u - \sqrt{-1} D_{n+j} u], \quad |\partial u|^2 = \langle \partial u, \partial u\rangle =\sum_{j=1}^n \partial_j u \overline{\partial_{j} u} = \frac{1}{4}|Du|^2 , \\
& \partial_{j\overline{j}} u  = \frac{1}{4}[D_{jj} u + D_{(n+j)(n+j)} u].
\end{align*}
\end{remark}

In particular, for $k=n$, we can obtain the estimate of double normal second derivatives, and obtain the existence theorem as follows.
\begin{theorem} \label{th1.3}
Suppose that $\Omega \subset \mathbb{C}^n$ is a $C^4$ strictly convex domain, $k =n$, $\nu$ is the outer unit normal vector of $\partial \Omega$, $\alpha_l(z) \in C^2(\overline{\Omega})$ with $l=0, 1, \cdots, n-1$ are positive real valued functions and $\varphi \in C^3(\partial \Omega)$ is a real valued function. Then there exists a unique plurisubharmonic solution $u \in C^{3, \alpha}(\overline \Omega)$ for the Neumann problem of mixed complex Hessian equations \eqref{1.2} with $k=n$.
\end{theorem}
Also, we can obtain the existence theorem for the corresponding classical Neumann problem of mixed complex Hessian equation with $k=n$.
\begin{theorem}\label{th1.4}
Suppose that $\Omega \subset \mathbb{C}^n$ is a $C^4$ strictly convex domain, $k= n$, $\nu$ is the outer unit normal vector of $\partial \Omega$, $\alpha_l(z) \in C^2(\overline{\Omega})$ with $l=0, 1, \cdots, n-1$ are positive real valued functions and $\varphi \in C^3(\partial \Omega)$ is a real valued function. Then there exists a unique constant $c$, such that the Neumann problem of mixed complex Hessian equations
\begin{align} \label{1.4}
\left\{ \begin{array}{l}
\sigma _n (\partial \bar{\partial} u) = \sum\limits_{l=0}^{n-1} \alpha_l(z) \sigma_l (\partial \bar{\partial} u),  \quad \text{in} \quad \Omega,\\
D_\nu u= c + \varphi(z), \quad \text{on} \quad \partial \Omega,
 \end{array} \right.
\end{align}
has plurisubharmonic solutions $u \in C^{3, \alpha}(\overline \Omega)$, which are unique up to a constant.
\end{theorem}

The rest of this paper is organized as follows. In Section 2, collect some properties of the elementary symmetric function $\sigma_k$, and establish some key lemmas. In Section 3, we prove Theorem \ref{th1.1}. In Section 4, we prove Theorem \ref{th1.3} and \ref{th1.4}.

\section{Preliminaries}

In this section, we give some basic properties of elementary symmetric functions, which could be found in
\cite{L96}, and establish some key lemmas.

\subsection{Basic properties of elementary symmetric functions}

First, we denote by $\sigma _m (\lambda \left| i \right.)$ the symmetric
function with $\lambda_i = 0$ and $\sigma _m (\lambda \left| ij
\right.)$ the symmetric function with $\lambda_i =\lambda_j = 0$.
\begin{proposition}\label{prop2.1}
Let $\lambda=(\lambda_1,\dots,\lambda_n)\in\mathbb{R}^n$ and $m
= 1, \cdots,n$, then
\begin{align*}
&\sigma_m(\lambda)=\sigma_m(\lambda|i)+\lambda_i\sigma_{m-1}(\lambda|i), \quad \forall \,1\leq i\leq n,\\
&\sum_i \lambda_i\sigma_{m-1}(\lambda|i)=m\sigma_{m}(\lambda),\\
&\sum_i\sigma_{m}(\lambda|i)=(n-m)\sigma_{m}(\lambda).
\end{align*}
\end{proposition}

We also denote by $\sigma _m (W \left|
i \right.)$ the symmetric function with $W$ deleting the $i$-row and
$i$-column and $\sigma _m (W \left| ij \right.)$ the symmetric
function with $W$ deleting the $i,j$-rows and $i,j$-columns. Then
we have the following identities.
\begin{proposition}\label{prop2.2}
Suppose $W=(W_{ij})$ is diagonal, and $m$ is a positive integer,
then
\begin{align*}
\frac{{\partial \sigma _m (W)}} {{\partial W_{ij} }} = \begin{cases}
\sigma _{m - 1} (W\left| i \right.), &\text{if } i = j, \\
0, &\text{if } i \ne j.
\end{cases}
\end{align*}
\end{proposition}

Recall that the G{\aa}rding's cone is defined as
\begin{equation}\label{2.1}
\Gamma_m  = \{ \lambda  \in \mathbb{R}^n :\sigma _i (\lambda ) >
0,\forall 1 \le i \le m\}.
\end{equation}

\begin{proposition}\label{prop2.3}
Let $\lambda=(\lambda_1,\dots,\lambda_n) \in \Gamma_m$ and $m \in \{1,2, \cdots, n\}$. Suppose that
$$
\lambda_1 \geq \cdots \geq \lambda_m \geq \cdots \geq \lambda_n,
$$
then we have
\begin{align}
\label{2.2}& \sigma_{m-1} (\lambda|n) \geq \sigma_{m-1} (\lambda|n-1) \geq \cdots \geq \sigma_{m-1} (\lambda|m) \geq \cdots \geq \sigma_{m-1} (\lambda|1) >0; \\
\label{2.3}& \lambda_1 \geq \cdots \geq \lambda_m >  0, \quad \sigma _m (\lambda)\leq C_n^m  \lambda_1 \cdots \lambda_m; \\
\label{2.4}& \lambda _1 \sigma _{m - 1} (\lambda |1) \geq \frac{m} {{n}}\sigma _m (\lambda); \\
\label{2.5}& \sigma _{m - 1} (\lambda |m) \geq c(n,m)\sigma_{m-1} (\lambda);
\end{align}
where $C_n^m = \frac{n!}{m! (n-m)!}$.
\end{proposition}

\begin{proof}
All the properties are well known. For example, see \cite{L96} or \cite{HS99} for a proof of \eqref{2.2},
\cite{L91} for \eqref{2.3}, \cite{CW01} or \cite{HMW10} for \eqref{2.4}, and \cite{LT94} for \eqref{2.5}.
\end{proof}

The generalized Newton-MacLaurin inequality is as follows, which will be used all the time.
\begin{proposition}\label{prop2.4}
For $\lambda \in \Gamma_m$ and $m > l \geq 0$, $ r > s \geq 0$, $m \geq r$, $l \geq s$, we have
\begin{align}
\Bigg[\frac{{\sigma _m (\lambda )}/{C_n^m }}{{\sigma _l (\lambda )}/{C_n^l }}\Bigg]^{\frac{1}{m-l}}
\le \Bigg[\frac{{\sigma _r (\lambda )}/{C_n^r }}{{\sigma _s (\lambda )}/{C_n^s }}\Bigg]^{\frac{1}{r-s}}. \notag
\end{align}
\end{proposition}
\begin{proof}
See \cite{S05}.
\end{proof}

\subsection{Key Lemmas}

In the establishment of a priori estimates, the following inequalities and properties play an important role.

For the convenience of notations, we will denote
\begin{equation}\label{2.6}
G_k(\partial \bar{\partial} u):= \frac{\sigma_k(\partial \bar{\partial} u)}{\sigma_{k-1}(\partial \bar{\partial} u)},\ \ G_l(\partial \bar{\partial} u) := -\frac{\sigma_l(\partial \bar{\partial} u)}{\sigma_{k-1}(\partial \bar{\partial} u)},~ 0\leq l\leq k-2,
\end{equation}
\begin{equation}\label{2.7}
G(\partial \bar{\partial} u, z):= G_k(\partial \bar{\partial} u) + \sum_{l=0}^{k-2} \alpha_l(z) G_l(\partial \bar{\partial} u),
\end{equation}
and
\begin{equation}\label{2.8}
G^{i\bar{j}} :=\frac{\partial G}{\partial u_{i\bar{j}}}, ~~ 1 \leq i, j \leq n.
\end{equation}

\begin{lemma}\label{lem2.5}
If $u$ is a $C^2$ function with $\lambda(\partial \bar{\partial} u)\in \Gamma_{k}$, and $\alpha_l(z)$ ($0 \leq l \leq k-2$) are positive, then the operator $G$ is elliptic and concave.
\end{lemma}
\begin{proof}
The proof is similar with the real case in \cite{GZ2019}.
\end{proof}

\begin{lemma}\label{lem2.6}
If $u$ is a $k$-admissible solution of \eqref{1.1}, and $\alpha_l(z)$ ($0 \leq l \leq k-1$) are positive, then
\begin{align}
\label{2.9}& 0 <\frac{\sigma_l(\partial \bar{\partial} u)}{\sigma_{k-1}(\partial \bar{\partial} u)} \leq C(n,k, \inf_\Omega \alpha_l), ~~ 0 \leq l \leq k-2; \\
\label{2.10}& 0< \inf_\Omega \alpha_{k-1} \leq \frac{\sigma_k(\partial \bar{\partial} u)}{\sigma_{k-1}(\partial \bar{\partial} u)} \leq C(n,k, \sum\limits_{l=0}^{k-1}\sup_\Omega \alpha_l).
\end{align}
\end{lemma}
\begin{proof}
The left hand sides of \eqref{2.9} and \eqref{2.10} are easy to prove. In the following, we prove the right hand sides.

Firstly, if $\frac{\sigma_k}{\sigma_{k-1}}\leq 1$, then we get from the equation \eqref{1.1}
\begin{equation*}
\alpha_l \frac{\sigma_l}{\sigma_{k-1}} \leq \frac{\sigma_k}{\sigma_{k-1}}  \leq 1, ~~0 \leq l \leq k-2.
\end{equation*}

Secondly, if $\frac{\sigma_k}{\sigma_{k-1}} > 1$, i.e. $\frac{\sigma_{k-1}}{\sigma_{k}} < 1$.
We can get for $0 \leq l \leq k-2$ by the Newton-MacLaurin inequality,
\begin{equation*}
\frac{\sigma_l}{\sigma_{k-1}}\leq \frac{(C_n^k)^{k-1-l}C_n^l}{(C_n^{k-1})^{k-l}}(\frac{\sigma_{k-1}}{\sigma_k})^{k-1-l} \leq \frac{(C_n^k)^{k-1-l}C_n^l}{(C_n^{k-1})^{k-l}} \leq C(n,k),
\end{equation*}
and
\begin{equation*}
\frac{\sigma_k}{\sigma_{k-1}} = \sum\limits_{l=0}^{k-1} \alpha_l \frac{\sigma_l}{\sigma_{k-1}} \leq  C(n,k) \sum\limits_{l=0}^{k-1} \sup_\Omega \alpha_l.
\end{equation*}
\end{proof}

\begin{lemma}\label{lem2.7}
If $u$ is a $k$-admissible solution of \eqref{1.1}, and $\alpha_l(z)$ ($0 \leq l \leq k-1$) are positive, then
\begin{align}
\label{2.11}& \frac{n-k+1}{k} \leq \sum G^{i\bar{i}} < n-k-1; \\
\label{2.12}&  \inf_\Omega \alpha_{k-1} \leq  \sum G^{i\bar{j}} u_{i\bar{j}} \leq C(n, k, \sum\limits_{l=0}^{k-1} \sup_\Omega \alpha_l).
\end{align}
\end{lemma}

\begin{proof}
By direct computations, we can get
\begin{align}
\sum {G^{i\bar{i}} }  \ge \sum {\frac{{\partial \left( {\frac{{\sigma _k }}{{\sigma _{k - 1} }}} \right)}}{{\partial \lambda _i }}}  =& \sum {\frac{{\sigma _{k - 1} (\lambda |i)\sigma _{k - 1}  - \sigma _k \sigma _{k - 2} (\lambda |i)}}{{\sigma _{k - 1} ^2 }}}  \notag \\
=& \frac{{(n - k + 1)\sigma _{k - 1} ^2  - (n - k + 2)\sigma _k \sigma _{k - 2} }}{{\sigma _{k - 1} ^2 }} \notag \\
\ge& \frac{{n - k + 1}}{k},
\end{align}
and
\begin{align}
\sum {G^{i\bar{i}} }  =& \sum {\frac{{\partial \left( {\frac{{\sigma _k }}{{\sigma _{k - 1} }}} \right)}}{{\partial \lambda _i }}}  - \sum\limits_{l = 0}^{k - 2} {\alpha _l \sum\limits_i {\frac{{\partial \left( {\frac{{\sigma _l }}{{\sigma _{k - 1} }}} \right)}}{{\partial \lambda _i }}} }  \notag \\
=& \sum {\frac{{\sigma _{k - 1} (\lambda |i)\sigma _{k - 1}  - \sigma _k \sigma _{k - 2} (\lambda |i)}}{{\sigma _{k - 1} ^2 }}}  - \sum\limits_{l = 0}^{k - 2} {\alpha _l \sum\limits_i {\frac{{\sigma _{l - 1} (\lambda |i)\sigma _{k - 1}  - \sigma _l \sigma _{k - 2} (\lambda |i)}}{{\sigma _{k - 1} ^2 }}} } \notag \\
=& \frac{{(n - k + 1)\sigma _{k - 1} ^2  - (n - k + 2)\sigma _k \sigma _{k - 2} }}{{\sigma _{k - 1} ^2 }} \notag \\
 &+ \sum\limits_{l = 0}^{k - 2} {\alpha _l \frac{{(n - k + 2)\sigma _l \sigma _{k - 2}  - (n - l + 1)\sigma _{l - 1} \sigma _{k - 1} }}{{\sigma _{k - 1} ^2 }}}  \notag \\
\le& (n - k + 1) - \frac{{(n - k + 2)\sigma _{k - 2} }}{{\sigma _{k - 1} }}\left( {\frac{{\sigma _k }}{{\sigma _{k - 1} }} - \sum\limits_{l = 0}^{k - 2} {\alpha _l \frac{{\sigma _l }}{{\sigma _{k - 1} }}} } \right)  \notag \\
<& n - k + 1,
\end{align}
hence \eqref{2.11} holds. Also, we can get
\begin{align}
\sum {G^{i\bar{j}} u_{i\bar{j}} }  =& \sum {\frac{{\partial \left( {\frac{{\sigma _k }}{{\sigma _{k - 1} }}} \right)}}{{\partial \lambda _i }}\lambda _i }  - \sum\limits_{l = 0}^{k - 2} {\alpha _l \sum\limits_i {\frac{{\partial \left( {\frac{{\sigma _l }}{{\sigma _{k - 1} }}} \right)}}{{\partial \lambda _i }}\lambda _i } }  \notag \\
=& \frac{{\sigma _k }}{{\sigma _{k - 1} }} + \sum\limits_{l = 0}^{k - 2} {(k - 1 - l)\alpha _l \frac{{\sigma _l }}{{\sigma _{k - 1} }}}  \notag \\
=& \alpha _{k-1}+ \sum\limits_{l = 0}^{k - 2} {(k  - l)\alpha _l \frac{{\sigma _l }}{{\sigma _{k - 1} }}},
\end{align}
hence \eqref{2.12} holds.

\end{proof}

\section{Proof of Theorem \ref{th1.1}}

In this section, we prove Theorem \ref{th1.1}.

\subsection{$C^0$ estimate}

The $C^0$ estimate is easy. For completeness, we produce a proof here following the idea of Lions-Trudinger-Urbas \cite{LTU86}.

\begin{theorem} \label{th3.1}
Suppose $\Omega \subset \mathbb{C}^n$ is a $C^1$ bounded domain, $\alpha_l(z) \in C^0(\overline{\Omega})$ with $l=0, 1, \cdots, k-1$ are positive functions and $\varphi \in C^0(\partial \Omega)$, and $u \in C^2(\Omega)\cap C^1(\overline \Omega)$ is the $k$-admissible solution of the equation \eqref{1.2} with $\varepsilon \in (0,1)$, then we have
\begin{align}\label{3.1}
\sup_\Omega |\varepsilon u|  \leq M_0,
\end{align}
where $M_0$ depends on $n$, $k$, $ \text{diam} (\Omega)$, $\max\limits_{\partial \Omega} |\varphi|$ and $\sum\limits_{l=0}^{k-1} \sup\limits_\Omega \alpha_l$.
\end{theorem}

\begin{proof}
Firstly, since $u$ is subharmonic, the maximum of $u$ is attained at some boundary point $z_0 \in \partial \Omega$. Then we can get
\begin{align}\label{3.2}
0 \leq D_\nu u (z_0) =-\varepsilon u(z_0) + \varphi (z_0).
\end{align}
Hence
\begin{align}\label{3.3}
\max_{\bar \Omega} ( \varepsilon u) = \varepsilon u(z_0) \leq \varphi (z_0) \leq \max_{\partial\Omega} |\varphi|.
\end{align}

For a fixed point $z_1 \in \Omega$, and a positive constant $A$ large enough, we have
\begin{align} \label{3.4}
G (\partial\bar{\partial} (A|z-z_1|^2), z) =& 2A \frac{C_n^k}{C_n^{k-1}} - \sum\limits_{l=0}^{k-2} \alpha_l (2A)^{-(k-1-l)} \frac{C_n^l}{C_n^{k-1}} \notag \\
\geq& \sup_\Omega \alpha_{k-1} \geq \alpha_{k-1}(z) = G (\partial\bar{\partial} u, z).
\end{align}
By the comparison principle, we know $u - A|z-z_1|^2$ attains its minimum at some boundary point $z_2 \in \partial \Omega$. Then
\begin{align}\label{3.5}
0 \geq& D_\nu (u - A|z-z_1|^2)|_{z=z_2} = D_\nu u(z_2) - A D_\nu(|z-z_1|^2)|_{z=z_2}  \notag \\
\geq& -\varepsilon u(z_2) - \max_{\partial \Omega} |\varphi| - 2A \text{diam}(\Omega).
\end{align}
Hence
\begin{align}\label{3.6}
\min_{\overline{\Omega}} (\varepsilon u) \geq \varepsilon \min_{\overline{\Omega}} (u - A|z-z_1|^2) \geq & \varepsilon u (z_2) - A|z_2-z_1|^2 \notag \\
\geq& - \max_{\partial \Omega} |\varphi| - 2A \text{diam}(\Omega)- A \text{diam}(\Omega)^2.
\end{align}

\end{proof}

\subsection{Global Gradient estimate}

In this subsection, we prove the global gradient estimate (independent of $\varepsilon$), using a similar argument of complex Monge-Amp\`ere equation in Li \cite{L94}.

\begin{theorem} \label{th3.2}
Suppose $\Omega \subset \mathbb{C}^n$ is a $C^3$ strictly convex domain, $\alpha_l(z) \in C^1(\overline{\Omega})$ with $l=0, 1, \cdots, k-1$ are positive functions and $\varphi \in C^2(\partial \Omega)$, and $u \in C^3(\Omega)\cap C^2(\overline \Omega)$ is the $k$-admissible solution of the equation \eqref{1.2} with $\varepsilon >0$ sufficiently small, then we have
\begin{align}\label{3.7}
\sup_\Omega |D u|  \leq M_1,
\end{align}
where $M_1$ depends on $n$, $k$, $\Omega$, $|\varphi|_{C^2}$, $\inf _\Omega\alpha_l$ and $|\alpha_l|_{C^1}$.
\end{theorem}

\begin{proof}

In order to prove \eqref{3.7}, it suffices to prove
\begin{align}\label{3.9}
D_{\xi}u(z) \leq  M_1, \quad \forall (z, \xi) \in \overline{\Omega }\times \mathbb{S}^{2n-1}.
\end{align}

For any $(z, \xi) \in \overline{\Omega }\times \mathbb{S}^{2n-1}$, denote
\begin{align}\label{4.3}
W(z, \xi)=D_{\xi}u(z)-\langle \nu, \xi \rangle (- \varepsilon u + \varphi(z))+ \varepsilon^2 u^2+K|z|^2,
\end{align}
where $K$ is a large constant to be determined later, and $\nu$ is a $C^2(\overline{\Omega})$ extension of the outer unit normal vector field on $\partial \Omega$.

 Assume $W$ achieves its maximum at $(z_0, \xi_0) \in \overline{\Omega }\times \mathbb{S}^{2n-1}$. It is easy to see $D_{\xi_0} u (z_0) >0$. We claim $z_0 \in \partial \Omega$. Otherwise, if $z_0 \in \Omega$, we shall get a contradiction in the following.

 Firstly, we rotate the coordinates such that $\partial \bar{\partial}u (z_0)$ is diagonal. It is easy to see $\{G^{i\bar{j}} \}$ is diagonal. For fixed $\xi = \xi_0$, $W(z, \xi_0)$ achieves its maximum at the same point $z_0 \in \Omega$ and we can easily get at $z_0$,
\begin{align}\label{4.4}
0\ge G^{i\bar{i}} \partial_{i\bar{i}} W=&G^{i\bar{i}}\big[ \partial_{i\bar{i}}D_{\xi_0}u -<\nu, \xi_0>_{i\bar{i}} (- \varepsilon u + \varphi)- <\nu, \xi_0> (- \varepsilon u_{i\bar{i}} + \varphi_{i\bar{i}}) \notag \\
&\quad \quad- <\nu, \xi_0>_{i} (- \varepsilon u_{\bar{i}} + \varphi_{\bar{i}})- <\nu, \xi_0>_{\bar{i}} (- \varepsilon u_i + \varphi_i)\notag\\
&\quad \quad +2 \varepsilon^2 u_i u_{\bar{i}}+2\varepsilon^2 u u_{i\bar{i}}+K \big] \notag \\
=& D_{\xi_0} \alpha_{k-1} + \sum\limits_{l=0}^{k-2} D_{\xi_0} \alpha_l \frac{\sigma_l}{\sigma_{k-1}} + G^{i\bar{i}}\big[2 \varepsilon^2  u_i u_{\bar{i}} +  \langle\nu, \xi_0\rangle_{i} \varepsilon u_{\bar{i}} +  \langle\nu, \xi_0\rangle_{\bar{i}} \varepsilon u_{i}\big] \notag \\
&+  G^{i\bar{i}} u_{i\bar{i}} \big[ \varepsilon \langle \nu, \xi_0 \rangle +2\varepsilon^2 u \big]  \notag \\
&+ G^{i\bar{i}}\big[K  -\langle\nu, \xi_0\rangle_{i\bar{i}} (- \varepsilon u + \varphi) - \langle\nu, \xi_0\rangle \varphi_{i\bar{i}} - \langle\nu, \xi_0\rangle_{i} \varphi_{\bar{i}} - \langle\nu, \xi_0\rangle_{\bar{i}} \varphi_{{i}} \big]  \notag \\
\geq& -|D \alpha_{k-1}|  - \sum\limits_{l=0}^{k-2} |D \alpha_l| C(n,k,\inf \alpha_l) - (n-k+1) |D \langle \nu, \xi_0\rangle|^2 \notag \\
&-  C(n, k, \sum \sup \alpha_l) \big[1 +2M_0 \big]  \notag \\
&+ \frac{n-k+1}{k} \big[K  - |D^2 \langle \nu, \xi_0 \rangle| (M_0 + |\varphi|) - |D^2 \varphi| - 2 |D \langle \nu, \xi_0 \rangle| |D\varphi| \big] \notag \\
>& 0,
\end{align}
where $K$ is large enough, depending only on $n$, $k$, $\Omega$, $M_0$, $|\varphi|_{C^2}$ and $\alpha_l$. This is a contradiction.

So $z_0 \in \partial \Omega$. Then we continue our proof into the following three cases.

(a) If $\xi_0$ is normal at $z_0\in \partial \Omega,$ then
\begin{align}
W(z_0,\xi_0)= \varepsilon ^2 u^2+ K |z_0|^2 \le C. \notag
\end{align}
Then we can easily obtain \eqref{3.9}.

(b) If $\xi_0$ is non-tangential at $z_0\in \partial \Omega$, then we can write $\xi_0=\alpha \tau + \beta \nu$,
where $\tau \in \mathbb{S}^{2n-1}$ is tangential at $x_0$, that is $\langle\tau, \nu\rangle=0$, $\alpha =\langle \xi_0, \tau \rangle >0$, $\beta= \langle \xi_0, \nu \rangle <1$, and $\alpha^2+\beta^2=1$. Then we have
\begin{align}\label{3.12}
W(z_0, \xi_0)=&\alpha D_{\tau}u + \varepsilon ^2 u^2 + K |z_0|^2 \notag \\
\le &\alpha W(z_0, \xi_0)+(1-\alpha)(\varepsilon ^2 u^2 + K |z_0|^2),
\end{align}
so
\begin{align}
W(z_0, \xi_0)\le  \varepsilon ^2 u^2 + K |z_0|^2  \leq C. \notag
\end{align}
Then we can easily get \eqref{3.9}.

(c) If $\xi_0$ is tangential at $x_0\in \partial \Omega$, we may assume that the outer normal direction of $\Omega$ at $z_0$ is $(0,\cdots,0,1)$. By a rotation, we assume that $\xi_0=(1, \cdots, 0)=e_1$. Then we have
\begin{align}
0 \leq& D_{\nu}W(z_0, \xi_0)\\=&D_{\nu}D_1u-D_{\nu}\langle\nu, \xi_0\rangle(- \varepsilon u + \varphi)+ 2 u \cdot D_{\nu}u+K D_{\nu} |z_0|^2  \notag \\
\leq& D_{\nu}D_{1}u + C_1 \notag\\
=& D_{1}D_{\nu}u - D_1\nu_kD_ku + C_1.
\end{align}
By the boundary condition, we know
\begin{align}
 D_{1}D_{\nu}u =  D_{1}(- \varepsilon u + \varphi) \leq  D_1 \varphi.
\end{align}
Following the argument of \cite{L94}, we can get
\begin{align}
 - D_1\nu_k D_k u  \leq  - \kappa_{min} W(z_0, \xi_0) + C_2,
\end{align}
where $\kappa_{min}$ is the minimum principal curvature of $\partial \Omega$. So
\begin{align}
 W(z_0, \xi_0) \leq \frac{C_1 + |D \varphi|+ C_2}{\kappa_{min}}.
\end{align}
Then we can conclude \eqref{3.9}.

\end{proof}

\subsection{Reduce global second derivatives to double normal second derivatives on the boundary}

In this subsection, we reduce global second derivatives to double normal second derivatives on the boundary, following the ideas of Lions-Trudinger-Urbas \cite{LTU86} and Li \cite{L94}.

\begin{theorem} \label{th3.3}
Suppose $\Omega \subset \mathbb{C}^n$ is a $C^4$ strictly convex domain, $\alpha_l(z) \in C^2(\overline{\Omega})$ with $l=0, 1, \cdots, k-1$ are positive functions and $\varphi \in C^3(\partial \Omega)$, and $u \in C^4(\Omega)\cap C^3(\overline \Omega)$ is the $k$-admissible solution of the equation \eqref{1.2} with $\varepsilon >0$ sufficiently small, then we have
\begin{align}\label{3.16}
\sup_{\Omega}|D^2u|\leq C(1+\max_{\partial \Omega}|D_{\nu\nu} u|),
\end{align}
where $C$ depends on $n$, $k$, $\Omega$, $|\varphi|_{C^3}$, $\inf_\Omega \alpha_l$ and $|\alpha_l|_{C^2}$.
\end{theorem}

\begin{proof}
Since $u$ is subharmonic, by the argument in \cite{L94}, we know that we only need to prove that
\begin{align}\label{3.17}
D_{\zeta\zeta}u(z)\le C (1+\max_{\partial \Omega} |D_{\nu \nu} u| ),  \quad \forall (z,\zeta)\in\bar{\Omega}\times \mathbb{S}^{2n-1}.
\end{align}
As the real case in \cite{LTU86}, we use the auxiliary function
\begin{align}
Q(z, \zeta)=D_{\zeta\zeta} u - v(z,\zeta) +|Du|^2 + K|z|^2,
\end{align}
where $v(z,\zeta)=2\langle\zeta, \nu\rangle \langle\zeta', D\varphi - \varepsilon D u - D_m u D\nu^m\rangle= a_m D_m u +b$, $\nu = (\nu^1, \nu^2, \cdots, \nu^{2n}) \in \mathbb{S}^{2n-1}$ is a $C^3(\overline{\Omega})$ extension of the outer unit normal vector field on $\partial \Omega$, $\zeta'=\zeta-\langle\zeta, \nu\rangle \nu$, $ a_m = - 2 \langle\zeta, \nu\rangle \langle\zeta', D \nu^m\rangle- 2\varepsilon \langle\zeta , \nu\rangle (\zeta')^m$, $b = 2 \langle\zeta, \nu \rangle \langle\zeta', D \varphi\rangle$, and $K>0$ is to be determined later.

For any $z \in \Omega$, we rotate the coordinates such that $\partial \bar{\partial} u (z)$ is diagonal, and then $\{G^{i\bar{j}} \}$ is diagonal. For any fixed $\zeta \in \mathbb{S}^{2n-1}$, we have
\begin{align}
G^{i\bar{i}} \partial_{i\bar{i}}Q=& G^{i\bar{i}} [\partial_{i\bar{i}}D_{\zeta\zeta}u-\partial_{i\bar{i}} a_m D_m u - a_m \partial_{i\bar{i}}D_m u -  \partial_i a_m \partial_{\bar{i}} D_m u -  \partial_{\bar{i}} a_m \partial_i D_m u\notag \\
& \qquad - \partial_{i\bar{i}} b+ 2 \partial_i D_m u \partial_{\bar{i}} D_m u +  2 D_m u \partial_{i\bar{i}} D_m u + K ] \notag \\
=& D_{\zeta\zeta}\alpha_{k-1} - 2\sum_{l=0}^{k-2} D_\zeta\alpha_l {G_l}^{ii} D_\zeta \partial_{i\bar{i}} u - \sum_{l=0}^{k-2} D_{\zeta\zeta}\alpha_{l} G_l - G^{i\bar{j},r\bar{s}} D_\zeta \partial_{i\bar{j}} u D_\zeta \partial_{r\bar{s}} u  \notag \\
&+ G^{i\bar{i}} [K - \partial_{i\bar{i}} a_m D_m u -\partial_{i\bar{i}} b ]+ G^{i\bar{i}} [2 \partial_i D_m u \partial_{\bar{i}} D_m u - \partial_i a_m \partial_{\bar{i}} D_m u -  \partial_{\bar{i}} a_m \partial_i D_m u]  \notag  \\
&+(-a_m+ 2 D_m u) [D_m \alpha_{k-1}- \sum_{l=0}^{k-2} D_m\alpha_l {G_l}]  \notag \\
\geq& G^{ii} [K -C_2]-C_1- 2\sum_{l=0}^{k-2} D_\zeta\alpha_l {G_l}^{i\bar{i}}  D_\zeta \partial_{i\bar{i}} u -G^{i\bar{j},r\bar{s}} D_\zeta \partial_{i\bar{j}} u D_\zeta \partial_{r\bar{s}} u   \notag \\
\geq& \frac{n-k+1}{k}[K -C_2]-C_1-C_3  >0,
\end{align}
where $K$ is large enough, and we used the fact
\begin{align}
&- 2\sum_{l=0}^{k-2} D_\zeta\alpha_l {G_l}^{i\bar{i}}  D_\zeta \partial_{i\bar{i}} u -G^{i\bar{j},r\bar{s}} D_\zeta \partial_{i\bar{j}} u D_\zeta \partial_{r\bar{s}} u \notag \\
\geq& - 2\sum_{l=0}^{k-2} D_\zeta\alpha_l {G_l}^{i\bar{i}}  D_\zeta \partial_{i\bar{i}} u - \sum_{l=0}^{k-2} \alpha_l {G_l}^{i\bar{j},r\bar{s}} D_\zeta \partial_{i\bar{j}} u D_\zeta \partial_{r\bar{s}} u  \notag \\
=&  - 2\sum_{l=0}^{k-2} D_\zeta\alpha_l {G_l}^{i\bar{i}}  D_\zeta \partial_{i\bar{i}} u \notag \\
 &- \sum_{l=0}^{k-2} \alpha_l \Big[\frac{{k - 1 - l}}{{\left( {\frac{{\sigma _{k - 1} }}{{\sigma _l }}} \right)^{\frac{{k - l}}{{k - 1 - l}}} }}\frac{{\partial ^2 \left( {\frac{{\sigma _{k - 1} }}{{\sigma _l }}} \right)^{\frac{1}{{k - 1 - l}}} }}{{\partial u_{i\bar{j}} \partial u_{r\bar{s}} }} - \frac{{k - l}}{{k - 1 - l}}\frac{1}{{G_l }} {G_l} ^{i\bar{j}} {G_l} ^{r\bar{s}}\Big]D_\zeta \partial_{i\bar{j}} u D_\zeta \partial_{r\bar{s}} u  \notag \\
\geq& \sum_{l=0}^{k-2} \frac{{k - 1 - l}}{{k - l}}\frac{{(D_\zeta\alpha_l ) ^2 }}{{\alpha _l }}G_l \notag \\
\geq& -C_3. \notag
\end{align}
So $\max\limits_{\overline{\Omega}} Q(z, \zeta)$ attains its maximum on $\partial \Omega$. Hence $\max\limits_{\overline{\Omega} \times \mathbb{S}^{2n-1}} Q(z, \zeta)$ attains its maximum at some point $z_0 \in \partial \Omega$ and some direction $\zeta_0 \in \mathbb{S}^{2n-1}$.

Then we continue our proof in the following two cases following the idea of \cite{L94}.

(a) If $\zeta_0$ is non-tangential at $z_0\in \partial \Omega$.

Then we can write $\zeta_0=\alpha\tau+\beta \nu$, where $\tau \in \mathbb{S}^{2n-1}$ is tangential at $z_0$, that is $\langle\tau, \nu\rangle=0$, $\alpha =\langle\zeta_0, \tau\rangle $, $\beta=\langle\zeta_0, \nu\rangle \ne 0$, and $\alpha^2+\beta^2=1$. Then we have
\begin{align}
D_{\zeta_0 \zeta_0} u (z_0) =& \alpha^2 D_{\tau \tau} u (z_0) + \beta^2 D_{\nu \nu} u (z_0)+ 2 \alpha \beta D_{\tau \nu} u (z_0) \notag \\
=& \alpha^2 D_{\tau \tau} u (z_0) + \beta^2 D_{\nu \nu} u (z_0) \notag\\
 &+ 2  (\xi_0 \cdot \nu)[\xi_0 - (\xi_0 \cdot \nu) \nu] [D \varphi - \varepsilon D u- D_m u D \nu^m], \notag
\end{align}
hence
\begin{align}
Q(z_0, \zeta_0) = \alpha^2 Q(z_0, \tau) + \beta^2 Q(z_0, \nu).
\end{align}
From the definition of $Q(z_0, \zeta_0)$, we know
\begin{align}
Q(z_0, \zeta_0)\le Q(z_0, \nu) \leq C(1+\max_{\partial \Omega} |D_{\nu \nu} u| ),
\end{align}
and we can prove \eqref{3.17}.

(b) If $\zeta_0$ is tangential at $z_0\in \partial \Omega$.

Then we have by Hopf Lemma
\begin{align}
0 \leq D_{\nu} Q(z_0, \zeta_0)=& D_{\nu}D_{\zeta_0 \zeta_0} u -D_{\nu}a_m D_m u -a_m D_{\nu} D_m u \notag \\
&- D_{\nu} b + 2 D_m u D_{\nu} D_m u+K D_{\nu} |z|^2  \notag \\
\leq& D_{\nu}D_{\zeta_0 \zeta_0} u+[ 2 D_m u- a_m]D_{\nu} D_m u + C_3.
\end{align}

By the boundary condition, we know
\begin{align}
D_{\nu}D_{\zeta_0 \zeta_0} u =&D_{\zeta_0 \zeta_0} D_{\nu}u -(D_{\zeta_0 \zeta_0}{\nu^m})D_m u-2 (D_{\zeta_0}{\nu^m}) D_{\zeta_0} D_m u \notag \\
=&D_{\zeta_0 \zeta_0} (- \varepsilon u + \varphi) -(D_{\zeta_0 \zeta_0}{\nu^m})D_m u-2 (D_{\zeta_0}{\nu^m}) D_{\zeta_0} D_m u \notag \\
\leq&- \varepsilon Q(z_0, \zeta_0)+C_4 -2 (D_{\zeta_0}{\nu^m}) D_{\zeta_0} D_m u.
\end{align}
Following the argument of \cite{L94}, we can get
\begin{align*}
&|D_{\nu} D_m u | \leq  C_5  (1+\max_{\partial \Omega} |D_{\nu \nu} u| ),\\
&-2 (D_{\zeta_0}{\nu^m}) D_{\zeta_0} D_m u \leq -  2 \kappa_{min} Q(z_0, \zeta_0)+C_6 (1+\max_{\partial \Omega} |D_{\nu \nu} u| ).
\end{align*}
So
\begin{align}
Q(z_0, \zeta_0) \leq \frac{C_{3} + C_{4}+ (2|Du| + |a_m|)C_{5}+ C_{6}}{2\kappa_{min}+ \varepsilon }(1+\max_{\partial \Omega} |D_{\nu \nu} u| ).
\end{align}
Then we can easily get \eqref{3.17}.

The proof is finished.
\end{proof}

\section{Proof of Theorem \ref{th1.3} and Theorem \ref{th1.4}}

 In this section, we prove Theorem \ref{th1.3} and Theorem \ref{th1.4}.

\subsection{Estimate of double normal second derivatives on boundary for $k =n$}

\begin{theorem} \label{th4.1}
Suppose $\Omega \subset \mathbb{C}^n$ is a $C^4$ strictly convex domain, $k=n$, $\alpha_l(z) \in C^2(\overline{\Omega})$ with $l=0, 1, \cdots, n-1$ are positive functions and $\varphi \in C^3(\partial \Omega)$, and $u \in C^4(\Omega)\cap C^3(\overline \Omega)$ is the plurisubharmonic solution of the equation \eqref{1.2} with $\varepsilon >0$ sufficiently small, then we have
\begin{align}\label{4.1}
\max_{\partial \Omega} |D_{\nu \nu} u|  \leq M_2,
\end{align}
where $M_2$ depends on $n$, $\Omega$, $|\varphi|_{C^3}$, $\inf _\Omega\alpha_l$ and $|\alpha_l|_{C^2}$..
\end{theorem}
\begin{proof}
Since $\Omega$ is a $C^4$ strictly convex domain, there is a strictly plurisubharmonic defining function $r \in C^4(\overline\Omega)$ such that
\begin{align}
&|D r| =1, \quad \text{ on } \partial \Omega, \\
&\partial \bar{\partial} r \geq k_0 I_n, \quad \text{ in } \overline{\Omega};
\end{align}
where $k_0$ is a positive constant depending only on $\Omega$, and $I_n$ is the $n \times n$ identity matrix.

Let $z_0\in\partial \Omega$ be an arbitrary point. By a shift and a rotation of the coordinates $\{z_1, \cdots, z_n\}$, we can assume that $z_0=0$, $\partial_{z_i} r (0)=0$ for $i<n$, and $D_{t_n} r(0) =-1$, $D_{t_{2n}} r(0) =0$. In $\overline{B}(0,\delta)\bigcap\overline{\Omega}$, a sufficiently small neighborhood of $z_0$, we can get by the Taylor expansion of $r$ up to second order
\begin{equation}
r(z)=-Re( z_n-\sum_{i,j=1}^n a_{ij} z_i z_j)+\sum_{i,j=1}^n  b_{i\bar{j}}z_i{\overline{z_j}}+O(|z|^3),
\end{equation}
where $\{ b_{i\bar{j}} \} = \partial \bar{\partial} r(0)$ is positive definite. We now introduce new coordinates $z'= \psi(z)$ of the form
\begin{equation}
z_i'= z_i, \quad \text{ for } i <n; \quad z_n'=z_n-\sum_{i,j=1}^n a_{ij} z_i z_j.
\end{equation}
In $\psi(\overline{B}(0,\delta)\bigcap\overline{\Omega})$, we have
\begin{equation}
r(z)|_{z = \psi^{-1}(z')}=-Re z_n'+\sum_{i,j=1}^n  b_{i\bar{j}}z_i'{\overline{z_j'}}+O(|z'|^3).
\end{equation}
Denote
\begin{align*}
&r_0(z')=-Re z_n'+\sum_{ij=1}^n  b_{i\bar{j}}z_i'{\overline{z_j'}}; \\
&B_j(z') = \big(\sum_{i=1}^n  [a_{ij} + a_{ji}] z_i \big)|_{z = \psi^{-1}(z')}, \quad j =1, \cdots, n; \\
&A_m(z') = B_m(1- B_n)^{-1} \frac{\partial r_0(z')}{\partial \overline{z_n'}}, \quad m =1, \cdots, n-1.
\end{align*}
It is easy to know $|B_j| = O(|z'|)$ for $j =1, \cdots, n$, and $A_j$ is holomorphic in $z' \in \psi(\overline{B}(0,\delta)\bigcap\overline{\Omega})$.
Following the calculations in \cite{L94}, we know the Neumann boundary condition in $z'$ coordinates
\begin{equation}
4 Re \left(<\partial_{z'} u,\partial_{z'} r_0>-\sum_{m=1}^{n-1}A_m\partial_{z_m'}u\right)=\phi(z', u)+O(|z'|^2)
\end{equation}
where $\phi(z', u)=|1-B_n(z')|^{-2}(-\varepsilon u+\varphi(z))|_{z = \psi^{-1}(z')}$.

Following the idea of \cite{L94}, we choose the auxiliary function
\begin{equation}
h(z')=4 Re[<\partial_{z'} u,\partial_{z'} r_0>-\sum_{m=1}^{n-1}A_m\partial_{z_m'}u]-\phi(z',u)+K r(z)|_{z = \psi^{-1}(z')}-K_1 Re(z_n'),
\end{equation}
where $K_1>0$ is sufficiently large such that
\begin{equation}\label{4.9}
h<0,\quad on\quad \psi\big(\partial(B(0, \delta)\cap \Omega)\setminus \partial \Omega\big),
\end{equation}
and
\begin{equation}\label{4.10}
h=-K_1 Re(z_n')+O(|z'|^2) \le 0\quad on\quad \psi\big(\partial(B(0, \delta)\cap \Omega)\cap\partial \Omega\big).
\end{equation}

Let
\begin{align}
G^{i \bar{j}} = \frac{\partial G }{\partial u_{z_i \overline{z_j}}}, \quad F^{i \bar{j}} = \frac{\partial  G}{\partial u_{z_i' \overline{z_j'}}}. \notag
\end{align}
It is easy to see
\begin{align}
F^{i \bar{j}} = G^{p \bar{q}} \big(\frac{\partial z_i'}{\partial z_p}\big) \overline{\big(\frac{\partial z_j'}{\partial z_q}\big)}. \notag
\end{align}

For any $z' \in \psi({B}(0,\delta)\bigcap{\Omega})$, we can get
\begin{align}
F^{i \bar{j}} \partial_{z_i' \overline{z_j'}} h =&  2F^{i\bar{j}}\partial_{z_i' \overline{z_j'}}\big(\partial_{z_m'} u \partial_{\overline{z_m'}} r_{0} + \partial_{\overline{z_m'}} u \partial_{z_m'} r_{0}\big) - 2F^{i\bar{j}}\partial_{z_i' \overline{z_j'}}\big(\sum_{m=1}^{n-1}A_m\partial_{z_m'}u+\overline{A_m} \partial_{\overline{z_m'}}u\big) \notag \\
&-F^{i \bar{j}} \partial_{z_i' \overline{z_j'}} \phi + K F^{i \bar{j}} \partial_{z_i' \overline{z_j'}} r  \notag \\
\geq&  2F^{i\bar{j}}\big(\partial_{z_m'} u_{z_i' \overline{z_j'}} \partial_{\overline{z_m'}} r_{0} + \partial_{\overline{z_m'}} u_{z_i' \overline{z_j'}} \partial_{z_m'} r_{0} + u_{z_m' \overline{z_j'}} \partial_{z_i' \overline{z_m'}} r_{0} +  u_{\overline{z_m'}z_i'} \partial_{z_m'\overline{z_j'}} r_{0}\big) \notag \\
&- 2\sum_{m=1}^{n-1}F^{i\bar{j}}\big(A_m\partial_{z_m'}u_{z_i' \overline{z_j'}}+\overline{A_m} \partial_{\overline{z_m'}}u_{z_i' \overline{z_j'}} + \partial_{z_i'}A_m u_{z_m' \overline{z_j'}}+\overline{\partial_{z_j'}A_m} u_{z_i' \overline{z_m'}}\big) \notag \\
& -G^{p \bar{q}} \partial_{z_p \overline{z_q}} \phi+ K G^{p \bar{q}} \partial_{z_p \overline{z_q}} r  \notag \\
\geq&  2G^{p\bar{q}}\big( u_{z_m' \overline{z_q}} \partial_{z_p \overline{z_m'}} r_{0} +  u_{\overline{z_m'}z_p} \partial_{z_m'\overline{z_q}} r_{0}\big)- 2\sum_{m=1}^{n-1}G^{p\bar{q}}\big( \partial_{z_p}A_m u_{z_m' \overline{z_q}}+\overline{\partial_{z_q}A_m} u_{z_p \overline{z_m'}}\big) \notag \\
& -C_7 + K k_0 \sum_{p=1}^n G^{p \bar{p}} \notag \\
\geq&  K k_0 \frac{n-k+1}{k} -C_8,
\end{align}
where $G^{p\bar{q}}u_{z_k' \overline{z_q}} = G^{p\bar{q}}u_{z_m \overline{z_q}} \frac{\partial z_m}{\partial z_k'}$ and $G^{p\bar{q}}u_{z_p \overline{z_k'}} = G^{p\bar{q}}u_{z_p \overline{z_m}} \overline{\big(\frac{\partial z_m}{\partial z_k'}\big)}$ are bounded by rotating the coordinates $\{z_1, \cdots, z_n\}$ such that $\partial_z \overline{\partial_z } u$ is diagonal.

Taking $K$ large enough, we can have $ G^{i \bar{j}} \partial_{z_i' \overline{z_j'}} h \geq 0$ in $\psi(B(0,\delta)\bigcap \Omega)$. By the maximum principle, we know $h(z')$ achieves its maximum at $\psi(\partial (B(0,\delta)\bigcap \Omega))$, and by \eqref{4.9} and \eqref{4.10} the maximum is attained at $z'=0$. Hence $h(z')|_{z' = \psi(z)}$ achieves its maximum at $z_0=0$.
Thus
\begin{align}
0\le D_{\nu}h(0)\le D_{\nu\nu}u(z_0)+C_{9}.
\end{align}
So we have $D_{\nu\nu}u(z_0) \ge -C_{9}.$

The same argument for
\begin{align}
\tilde{h}(z')=4Re[<\partial_{z'} u, \partial_{z'} r_0>-\sum_{m=1}^{n-1}A_m\partial_{z_m'}u]- \phi(z', u)- K r(z)|_{z = \psi^{-1}(z')} + K_1 Re(z_n') \notag
\end{align}
can give
\begin{align}
D_{\nu\nu}u(z_0) \le C_{10}.
\end{align}
This completes the estimates of the double normal derivative on the boundary.
\end{proof}

\subsection{Proof of Theorem \ref{th1.3}}

For $k=n$ and $\varepsilon >0$ sufficiently small, we have established the global $C^2$ estimates for the plurisubharmonic solution of the Neumann problem of mixed complex Hessian equation \eqref{1.2} in Section 3 and Subsection 4.1. By the global $C^2$ a priori estimates, we obtain that the equation \eqref{1.2} are uniformly elliptic in $\overline \Omega$.
Due to the concavity of the operator $G$, we can get the global H\"{o}lder estimates of second derivative following the discussions in \cite{LT86}, that is, we can get
\begin{align}
|u|_{C^{2,\alpha}(\overline \Omega)} \leq C,
\end{align}
where $C$ and $\alpha$ depend on $n$, $\Omega$, $\varepsilon$, $\inf \alpha_l$, $|\alpha_l|_{C^2}$ and $|\varphi|_{C^3}$. Then the $C^{3,\alpha}(\overline \Omega)$ estimates hold by differentiating the equation \eqref{1.2} and applying the Schauder theory for linear, uniformly elliptic equations.

Applying the method of continuity (see \cite{GT}, Theorem 17.28), we can show the existence of the  plurisubharmonic solution, and the solution is unique by Hopf lemma. By the standard regularity theory of uniformly elliptic partial differential equations, we can obtain the high order regularity.

\subsection{Proof of Theorem \ref{th1.4}}

By the argument in Subsection 4.2, we know there exists a unique plurisubharmonic solution $u^\varepsilon \in C^{3, \alpha}(\overline \Omega)$ to \eqref{1.2} for any small $\varepsilon >0$. Let $v^\varepsilon =u^\varepsilon - \frac{1}{|\Omega|} \int_\Omega u^\varepsilon$, and it is easy to know $v^\varepsilon$ satisfies
\begin{align}
\left\{ \begin{array}{l}
\sigma _n (\partial \bar{\partial} v^\varepsilon) = \sum\limits_{l=0}^{n-1} \alpha_l \sigma _l (\partial \bar{\partial}  v^\varepsilon),  \quad \text{in} \quad \Omega,\\
D_\nu (v^\varepsilon) = - \varepsilon v^\varepsilon - \frac{1}{|\Omega|} \int_\Omega \varepsilon u^\varepsilon + \varphi(x),\quad \text{on} \quad \partial \Omega.
 \end{array} \right.
\end{align}
By the global gradient estimate \eqref{3.7}, it is easy to know $\varepsilon \sup |D u^\varepsilon | \rightarrow 0$. Hence there is a constant $c$ and a function $v \in C^{2}(\overline \Omega)$, such that $-\varepsilon u^\varepsilon \rightarrow c$, $-\varepsilon v^\varepsilon \rightarrow 0$, $-\frac{1}{|\Omega|} \int_\Omega \varepsilon u^\varepsilon \rightarrow c$ and $v^\varepsilon \rightarrow v$ uniformly in $C^{2}(\overline \Omega)$ as $\varepsilon \rightarrow 0$. It is easy to verify that $v$ is a plurisubharmonic solution of
\begin{align}
\left\{ \begin{array}{l}
\sigma _n (\partial \bar{\partial}  v) = \sum\limits_{l=0}^{n-1} \alpha_l \sigma _l (\partial \bar{\partial}  v),  \quad \text{in} \quad \Omega,\\
D_\nu v = c + \varphi(x),\qquad \text{on} \quad \partial \Omega.
 \end{array} \right.
\end{align}
If there is another plurisubharmonic function $v_1 \in C^{2}(\overline \Omega)$ and another constant $c_1$ such that
\begin{align}
\left\{ \begin{array}{l}
\sigma _n (\partial \bar{\partial}  v_1) = \sum\limits_{l=0}^{n-1} \alpha_l \sigma _l (\partial \bar{\partial}  v_1),  \quad \text{in} \quad \Omega,\\
D_\nu v_1= c_1 + \varphi(x),\qquad \text{on} \quad \partial \Omega.
 \end{array} \right.
\end{align}
Applying the maximum principle and Hopf Lemma, we can know $c = c_1$ and $v - v_1$ is a constant.
By the standard regularity theory of uniformly elliptic partial differential equations, we can obtain the high oder regularity.


\begin{thebibliography}{50}

\bibitem{CNSK85}
L. Caffarelli, J. Kohn, L. Nirenberg, J. Spruck. The Dirichlet problem for nonlinear second-order elliptic equations. II. Complex Monge-Amp\`{e}re and uniformly elliptic equations. Comm. Pure Applied Math., 38 (1985), 209-252.

\bibitem{CNS84}
L. Caffarelli, L. Nirenberg, J. Spruck. Dirichlet problem for nonlinear second order
elliptic equations I, Monge-Amp\`{e}re equations. Comm. Pure Appl. Math., 37(1984), 369-402.

\bibitem{CNS85}
L. Caffarelli, L. Nirenberg, J. Spruck. Dirichlet problem for nonlinear second order
elliptic equations III, Functions of the eigenvalues of the Hessian. Acta Math.,
155(1985), 261-301.
%
%\bibitem{C15}
%C.Q. Chen. The interior gradient estimate of Hessian quotient equations. J. Differential Equations, 259(2015), 1014-1023.

\bibitem{CCX19}
C.Q. Chen, L. Chen, N. Xiang. The Classical Neumann Problem for a class of mixed Hessian equations. preprint, 2019.

\bibitem{CW17}
C.Q. Chen, W. Wei. The Neumann problem of complex Hessian quotient equations. preprint, 2017.

\bibitem{CZ16}
C.Q. Chen, D.K. Zhang. The Neumann problem of Hessian quotient equations. preprint, 2016.

\bibitem{CW01}
K.S. Chou, X.J. Wang. A variation theory of the Hessian equation. Comm. Pure Appl. Math., 54(2001), 1029-1064.

\bibitem{FY2007}
J.X. Fu, S.T. Yau. A Monge-Amp\`ere type equation motivated by string theorey. Comm. Anal. Geom., 15(2007), 29-76.

\bibitem{FY2008}
J.X. Fu, S.T. Yau. The theory of superstring with flux on non-K$\ddot{a}$hler manifolds and the complex Monge-Amp\`ere equation. J. Diff. Geom., 78(2008), 369-428.

\bibitem{GT}
D. Gilbarg, N. Trudinger. Elliptic Partial Differential Equations of Second Order.
Grundlehren der Mathematischen Wissenschaften, Vol. 224. Springer-Verlag,
Berlin-New York, 1977. x+401 pp. ISBN: 3-540-08007-4.

\bibitem{GZ2019}
P.F. Guan, X.W. Zhang. A class of curvature type equations. Pure and Applied Math Quarterly, to appear.

\bibitem{HL82}
R. Harvey, B. Lawson. Calibrated geometries. Acta Math., 148(1982), 47-157.

\bibitem{HMW10}
 Z.L. Hou, X.N. Ma, D.M. Wu. A second order estimate for complex Hessian equations
on a compact K\"{a}hler manifold. Math. Res. Lett., 17(2010), 3: 547-561.

\bibitem{HS99}
G. Huisken, C. Sinestrari. Convexity estimates for mean curvature flow and singularities
of mean convex surfaces. Acta Math., 183(1999), 1: 45-70.

\bibitem{I87}
N. Ivochkina. Solutions of the Dirichlet problem for certain equations of Monge-Amp\`{e}re type (in Russian). Mat. Sb., 128 (1985), 403-415: English translation in Math. USSR Sb.,56(1987).

\bibitem{JT15}
F.D. Jiang, N. Trudinger. Oblique boundary value problems for augmented Hessian equations I. Bulletin of Mathematical Sciences, 8(2018), 353-411.

\bibitem{JT16}
F.D. Jiang, N. Trudinger. Oblique boundary value problems for augmented Hessian equations II. Nonlinear Analysis: Theory, Methods $\&$ Applications,  154(2017), 148-173.

\bibitem{JT3}
F.D. Jiang, N. Trudinger. Oblique boundary value problems for augmented Hessian equation III. Comm. Part. Diff. Equa., 44(2019), 708-748.

\bibitem{Kr}
N.V. Krylov. On the general notion of fully nonlinear second order elliptic equation. Trans. Amer. Math. Soc., 3(1995), 857-895.

\bibitem{L94}
 S.Y. Li. On the Neumann problems for Complex Monge-Amp\`{e}re equations. Indiana Univ. Math. J., 43(1994), 1099-1122.

\bibitem{L91}
Y.Y. Li. Interior gradient estimates for solutions of certain fully nonlinear elliptic equations. J. Diff. Equa., 90(1991), 172-185.

\bibitem{L96}
G. Lieberman. Second order parabolic differential equations. World Scientific, 1996.

\bibitem{L13}
G. Lieberman. Oblique boundary value problems for elliptic equations. World Scientific Publishing, 2013.

\bibitem{LT86}
G. Lieberman, N. Trudinger. Nonlinear oblique boundary value problems for nonlinear
elliptic equations. Trans. Amer. Math. Soc., 295 (1986), 2: 509-546.

\bibitem{LT94}
M. Lin, N.S. Trudinger. On some inequalities for elementary symmetric functions. Bull.
Austral. Math. Soc., 50(1994), 317-326.

\bibitem{LTU86}
P.L. Lions, N. Trudinger, J. Urbas. The Neumann problem for equations of
Monge-Amp\`{e}re type. Comm. Pure Appl. Math., 39(1986), 539-563.

\bibitem{MQ15}
X.N. Ma, G.H. Qiu. The Neumann Problem for Hessian Equations. Comm. Math. Phys.,
366(2019), 1-28.

\bibitem{QX16}
G.H. Qiu, C. Xia. Classical Neumann Problems for Hessian equations and Alexandrov-Fenchel¡¯s inequalities. International Mathematics Research Notices, rnx296, 2018.


\bibitem{S}
R. Schneider. Convex bodies: The Brunn-Minkowski theory. Cambridge University, 1993.

\bibitem{S05}
J. Spruck. Geometric aspects of the theory of fully nonlinear elliptic equations. Clay Mathematics
Proceedings, volume 2, 2005, 283-309.

\bibitem{T87}
N. Trudinger. On degenerate fully nonlinear elliptic equations in balls. Bull. Aust. Math. Soc., 35 (1987), 299-307.

\bibitem{T95}
N. Trudinger. On the Dirichlet problem for Hessian equations. Acta Math.,
175(1995), 151-164.

\bibitem{U95}
J. Urbas. Nonlinear oblique boundary value problems for Hessian equations in two
dimensions. Ann. Inst Henri Poincar\`{e}-Anal. Non Lin., 12(1995), 507-575.

\bibitem{U96}
J. Urbas. Nonlinear oblique boundary value problems for two-dimensional curvature
equations. Adv. Diff. Equa., 1(1996), 3: 301-336.


\end{thebibliography}
\end{document}